\documentclass[12pt]{amsart}
\usepackage[all]{xy}
\usepackage{amsfonts}
\usepackage{amsthm, amsmath, amssymb}
\usepackage{array}
\usepackage{euscript}
\usepackage{epsfig, psfrag, epic}
\usepackage{graphicx}
\usepackage{xy}
\usepackage{multirow}
\usepackage{subfigure}

\title{Plumbers' Knots} 

\author{Chad Giusti}
\address{Warren Center for Network and Data Sciences\\University of Pennsylvania}
\email{cgiusti@seas.upenn.edu}

\newtheorem{thm}{Theorem}[section]
\newtheorem{cor}[thm]{Corollary}
\newtheorem{lem}[thm]{Lemma}
\newtheorem{prop}[thm]{Proposition}

\newtheorem{rmk}[thm]{Remark}
\newtheorem{defn}[thm]{Definition}

\newcommand{\into}{\hookrightarrow}
\newcommand{\dirlim}{\underrightarrow\lim}

\newcommand{\R}{{\mathbb R}}
\newcommand{\Z}{{\mathbb Z}}
\newcommand{\I}{{\mathbb I}}
\newcommand{\N}{{\mathbb N}}

\newcommand{\x}{{\mathbf x}}
\newcommand{\y}{{\mathbf y}}
\newcommand{\z}{{\mathbf z}}

\newcommand{\iI}{{\sc{Int}\;\I^3}}
\newcommand{\Cell}{\textrm{C}}

\def\co{\colon\thinspace}

\begin{document}
\baselineskip=17pt 
\begin{abstract}
We construct a new type of finite-complexity knot theory, the theory of plumbers' knots, which models classical knot theory and carries an intrinsic combinatorial cell structure for each fixed level of complexity. Utilizing these cell structures, describe algorithmic solutions to the problems of distinguishing and enumerating such knots.
\end{abstract}

\maketitle


In this paper, we introduce a new flavor of finite-complexity knot theory. The curves we consider, called \emph{plumbers' curves}, are a family of PL curves in $\I^3 \subset \mathbb{R}^3$ whose segments run parallel to the coordinate axes in the fixed order $\{x, y, z, x, y,\dots, z\}$. By focusing on such rigid curves, we maintain a thorough geometric understanding of the spaces without sacrificing the theoretical power of more general models.

To wit, the collection of all such curves decomposes into a directed system of spaces $P_n$ of \emph{plumbers' curves of $n$ moves}, each of which is homeomorphic to a Euclidean cube. Inside of each $P_n$ lies the subspace of \emph{plumbers' knots of $n$ moves}, $K_n$, whose weak homotopy type we show converges to that of the space of $C^1$ long knots. As such, these spaces provide us with a model for classical knot theory.

In addition, we show that each plumbers' curve space $P_n$ admits a cellular decomposition, and the induced cell complex can be canonically described through elementary combinatorics. Within this cell complex we identify a subcomplex for the discriminant, $S_n = P_n \setminus K_n$, providing an explicit (albeit still complicated) description of the geometry of the discriminant and how it evolves through the directed system. Other finite-complexity knot spaces do not seem to carry such structure, which make plumbers' knots very attractive for use in studying the Vassiliev spectral sequence, a program which we continue in later work. 

Understanding these decompositions allows us to construct a deterministic finite-time algorithm that enumerates the components of $K_n$. An implementation of this algorithm has demonstrated that there are seven knot types in $K_5$, forty-nine in $K_6$ and one thousand and eight in $K_7$ The seven components of $K_5$ correspond, topologically, to the unknot, three right-handed trefoils and three left-handed trefoils. This phenomenon of ``stuck knots", knots which are not isotopic at a level of finite complexity but whose topological isotopy classes coincide, was observed for PL knots by Calvo \cite{Calvo}, and in a strong sense lies at the heart of the problem of understanding these spaces.

We also note that plumbers' knots bear strong a resemblance to both lattice knots as considered, for example, in \cite{vanRensProm}, and the cube diagrams studied in \cite{Cube}. We observe that a plumbers' knot can be viewed as an approximation of a particular lattice knot, and put bounds on number of the classes of lattice knots that arise in each $P_n$. The similarity to the cube diagrams of Baldridge and Lowrance \cite{Cube} suggests that the tools we develop for plumbers' knots could be useful in the study of knot Floer homology.

The author would like to thank his advisor, Dev Sinha, for the countless hours of conversation and support without which this research would not have been possible.

\tableofcontents

\section{Plumbers' curves}
We begin by introducing plumbers' curves and the terminiology we will use to study them. 

Fix the standard orthonormal basis $\{x, y, z\}$ for $\R^3$. Let $\I^3$ be the cube $[0,1]^3\subseteq\mathbb{R}^3$. Given a point $v\in (\iI)^{n-1}$, we construct a map $\phi_{v}\co\R \to \R^3$ with support on the interval $[0, 1]$ which we call a {\em plumbers' curve}. 

\begin{defn}
Let $v=(v_1,\dots,v_{n-1}) \in (\iI)^{n-1}$, and by convention write $v_0 = (0, 0, 0)$ and $v_n = (1, 1, 1)$. For each $i\in\{0,\dots ,n-1\}$, define a sequence of three linear maps $\x_i^v(t), \y_i^v(t)$ and $\z_i^v(t)$ which interpolate between $v_{i}$ and $v_{i+1}$ by traveling continuously, in that order, parallel to the coordinate axes. Explicitly, define $\x_i^v(t) \co \left[\frac{i}{n}, \frac{i + 1}{n}\right] \to \R^{3}$ by 
\begin{align*}
\x_i^v(t) = \left((i+ 1 - nt){v}_i^x + (nt - i)  {v}_{i+1}^x,  {v}_i^y,  {v}_i^z\right)
\end{align*}
and $\y_i^v(t)$ and $\z_i^v(t)$ analogously. When the point $v$ is clear from context, we will suppress it from notation. 
We call each such map a \emph{pipe}, and each triple of consecutive pipes $(\x_i(t), \y_i(t), \z_i(t))$ for a given $i$ is a \emph{move}. Say the \emph{length} of $\x_i$ is $||\x_i|| = |\x_i(\frac{i}{n}) - \x_i(\frac{i+1}{n})|$, and its \emph{direction}  is
\begin{eqnarray*}
s(\x_i) = \left\{\begin{array}{ll}
 1& \x_i(\frac{i}{n}) > \x_i(\frac{i+1}{n})\\
 0& \x_i(\frac{i}{n}) = \x_i(\frac{i+1}{n})\\
 -1& \x_i(\frac{i}{n}) < \x_i(\frac{i+1}{n}) .
\end{array}\right .
\end{eqnarray*} 
Again, make appropriate modifications to apply these definitions to $\y_i$ and $\z_i$.

Let $\phi_{v}(t) \co [0,1] \to \I^{3}$ be the union of these maps. We call $\phi_{v}$ an \emph{$n$-move plumbers' curve}. (See Figure \ref{fig:aknot} for an example.)
\end{defn}

\begin{figure} 
\begin{center}
\includegraphics[width=7cm]{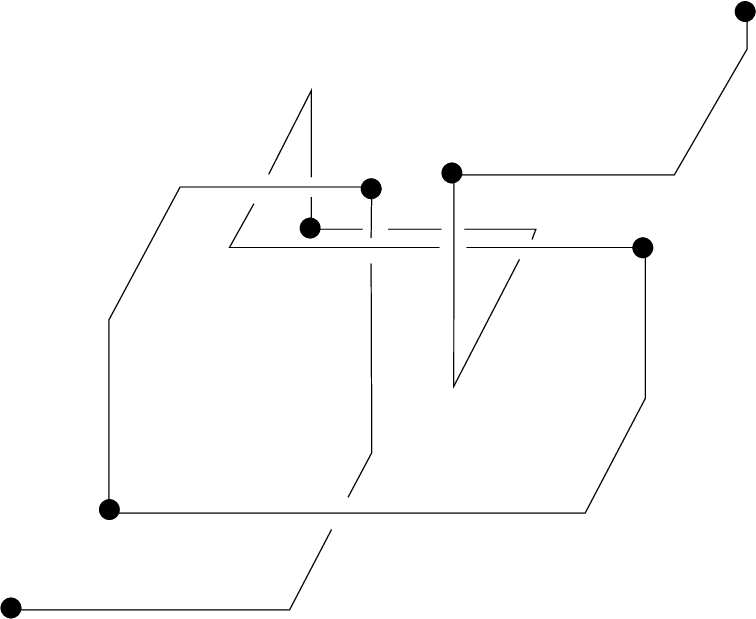}
\caption{A plumbers' curve of 6 moves.}
\label{fig:aknot}
\end{center}
\end{figure}

These maps encode piecewise linear motion parallel to the $x$-, $y$- and $z$-axes. A generic curve alternates these three directions of motion in the order $x, y, z, x,\dots,z$, however there is no restriction from pipes having length zero. Indeed, allowing such ``pauses'' and even allowing collections of vertices to coincide will be vital to our construction of a cell structure and a coherent directed system of the spaces of all such curves

\begin{defn}
Denote by $P_n$ the space of $n$-move plumbers maps, 
\begin{equation*}
P_n = \{\phi_v \;|\; v \in (\iI)^{n-1}\}.
\end{equation*}
\end{defn}

Observe that, up to reparametrization, any PL map with such an alternating image is produced by such a construction, so $P_n$ is a deformation retract of the collection of all such maps.

\subsection{Plumbers' knots}

There are two types of finite codimension singularities for classical $C^1$ curves: families of points at which the derivative vanishes and multiple points in the image. The rigid nature of plumbers' curves allows us to collapse all types of singularities to special cases of a particular condition the pipes, which we will readily be able to detect.

\begin{defn}  We say two pipes of a plumbers' curve are \emph{distant} if they are separated by at least three intervening pipes. 
\end{defn}

For example, $\x_i$ and $\x_{i+1}$ are not distant, but $\x_i$ and $\y_{i+1}$ are. 

\begin{defn}\label{def:singular}
A plumbers' curve is \emph{singular} if the images of any pair of its distant pipes intersect, and is \emph{non-singular} if not. 
\end{defn}

When discussing singularity of curves, we will sometimes abuse terminology by talking about pipes intersecting, rather than their images.

In Theorem \ref{thm:htpytypecorrect}, we will show that the weak homotopy type of the spaces of non-singular curves converges in the limit to that of the space of $C^1$ long knots, which will give us a rigorous justification for this definition. In the interim, to see that the intersection of distant pipes is an intuitively reasonable notion of singularity for plumbers' curves, we compare to the classical case. 

At a point with vanishing derivative in a $C^1$ curve, the $x$-, $y$- and $z$-partial derivatives are simultaneously zero. If we construct a plumbers' map with three consecutive coordinate equalities, for example $v_i^y = v_{i+1}^y$, $v_i^z = v_{i+1}^z$, $v_{i+1}^x = v_{i+2}^x$, then we can see that image of the ``head'' of the pipe $\x_i(t)$ is coincident with the images of the three zero-length pipes $\y_{i}(t)$, $\z_{i}(t)$ and $\x_{i+1}(t)$ that follow, and thus with the ``tail'' of the pipe $\y_{i+1}(t)$, which is distant. Any shorter sequence of consecutive zero-length pipes does not create an intersection of distant pipes, and intuitively corresponds to a nonsingular curve which is moving in a coordinate plane or parallel to a coordinate axis.

The other type of singularity, a multiple point in the image of the map, clearly corresponds to an intersection of of distant pipes, as well. 

\begin{rmk} \label{rmk:overrun} In the plumbers' curve setting, it may appear that there are other types of finite-condimension singularities not considered under Definition \ref{def:singular}. Suppose, for example, a plumbers' curve has a pair of parallel pipes whose images coincide along some interval in their domain. However, we observe that any such intersection has the property that at least one endpoint of one of the coincident pipes is contained in the image of the other pipe. Thus, there is also an intersection of non-parallel, necessarily distant pipes which occurs in addition to the originally observed ``overrun'' singularity.

For example, if the image of $\x_i(t)$ lies entirely in the image of $\x_j(t)$, then the image of $\z_{i-1}(t)$ and that of $\y_{i}(t)$ will also intersect the image of $\x_{j}(t)$. Further, at least one of those pipes must be distant to $\x_j(t)$, so this map is singular. 
\end{rmk} 

Now, we are prepared to define our primary objects of study: spaces of plumbers' knots and their discriminants. 

\begin{defn}
A non-singular plumbers' curve of $n$ moves is a \emph{plumbers' knot}, and the space of all such is $K_n \subseteq P_n$. The \emph{discriminant} $S_n = P_n \setminus K_n$ is the subspace of singular plumbers' curves. 
\end{defn}

In order to properly compare this theroy in classical knot theory, which deals with isotopy classes of non-singular $C^1$ curves, we require a notion of isotopy for plumbers' knots.

\begin{defn}
Let $\phi_v, \phi_w\in K_n$. We say $\phi_v$ and $\phi_w$ are \emph{geometrically isotopic} if there is a path $\Phi_{v, w} \co \I \to K_n$ with $\Phi_{v,w}(0) = v$ and $\Phi_{v,w}(1) = w$.
\end{defn}

Later, we will see that the notion geometric isotopy in any $K_n$ is stronger than the usual topological notion of knot isotopy. However, the two notions converge as we increase the articulation of plumbers' knots. The fact that geometric isotopies preserve the number of moves in a knot will be the fundamental component in our solutions of the problems of enumerating and distinguishing plumbers' knots.

\section{A cell complex for $S_n$}

The spaces $P_n$ possess an intrinsic combinatorial cell structure which is compatable with the partition into subspaces $S_n$ and $K_n$. The cell structure arises by collecting plumbers' curves $\phi_v$, into families based on the ordering of their vertices when projected onto the $x$, $y$ and $z$ axes. 

Let $\Sigma_n$ be the symmetric group on $n$ letters. It will sometimes be convenient to represent $\sigma\in\Sigma_n$ in \emph{permutation} notation as a string of integers, $\sigma(1)\sigma(2)\dots\sigma(n)$. By convention, we will use $\tau$ exclusively to refer to transpositions, while other symbols may represent any permutation.

Throughout, most of our constructions and results will be identical for $x$, $y$ and $z$ coordinates. To simplify statements, we will usually only explicitly describe the $x$ coordinate version.

\begin{defn} Let $\sigma \in \Sigma_{n-1}$. We say $\phi_v \in P_n$ \emph{respects $\sigma$ in x} if 
\begin{equation}
\label{eqn:respectsigma}
0 < v^x_{\sigma_i(1)} < v^x_{\sigma_i(2)} < \cdots < v^x_{\sigma_i(n-1)} < 1
\end{equation}
\end{defn}

Let $\sigma_x, \sigma_y, \sigma_z\in \Sigma_{n-1}$. Define an open set $e(\sigma_x, \sigma_y, \sigma_z) \subset P_n$ by 
\begin{eqnarray*}
e(\sigma_x, \sigma_y, \sigma_z)&=&\{\phi_v\in P_n \;|\;\phi_v\textrm{ respects }\sigma_x\textrm{  in }x, \sigma_y\textrm{  in }y\textrm{  and }\sigma_z\textrm{  in }z\}
\end{eqnarray*}
and denote by $\bar{e}(\sigma_x, \sigma_y, \sigma_z)$ its closure. These sets $\bar{e}(\sigma_x, \sigma_y, \sigma_z)$ are by definition products of three closed $(n-1)$-simplices. More specifically,  the family $\{\bar{e}(\sigma_x, \sigma_y, \sigma_z)\; |\; \sigma_x,\sigma_y, \sigma_z \in \Sigma_{n-1}\}$ is precisely the product of three copies of the standard decomposition of the cube $\I^{n-1}$  by (n-1)-simplices. Thus, these regions induce a cell structure on $P_n$ whose elements are products of the open simplices of various dimensions arising in the standard decomposition of $\I^{n-1}$. 

\begin{defn}
Let $\Cell_\bullet(P_n)$ be the chain complex associated to the cell structure induced by the collection $\{e(\sigma_x, \sigma_y, \sigma_z) \;|\; \sigma_x, \sigma_y, \sigma_z \in \Sigma_{n-1}\}$. By abuse, we write $e(\sigma_x, \sigma_y, \sigma_z)$ for elements of $\Cell_{3n-3}(P_n)$ as well as the corresponding subsets of $P_n$.
\end{defn}

It is clear from the geometry of the complex that the codimension of a cell is precisely the number of coordinate equalities in the points $v$ which compose the cell, and that all of the points in a given cell share the same collection of coordinate equalities and inequalities.

Each top-dimensional cell described above is clearly non-empty. By distinguishing a representative map from each cell, we can move toward ``discretizing'' this theory:

\begin{defn} \label{def:rep_knot} Let  $e=e(\sigma_x, \sigma_y, \sigma_z) \in \Cell_{3n-3}(P_n)$. The \emph{representative plumbers' knot for $e$} is the $n$-move plumbers knot $\phi_{v(e)}$ where $v(e)$ is given by
\begin{displaymath}
v(e)_i=\left(\frac{\sigma^{-1}_x(i)}{n}, \frac{\sigma^{-1}_y(i)}{n}, \frac{\sigma^{-1}_z(i)}{n}\right).
\end{displaymath}

\end{defn}

Note that every curve described in Definition \ref{def:rep_knot} is, indeed, non-singular because no pair of distant pipes in the image can be coplanar due to coordinate inequalities.

Clearly, any plumbers' knot in the closure of a top dimensional cell is geometrically isotopic to the representative knot for that cell via a straight line geometric isotopy. Thus, to study geometric isotopy types of $n$-move plumbers' knots it suffices to study only the classes representative knots of top-dimensional cells.  

\begin{lem} \label{lem:singinface} If a plumbers' curve $\phi_v$ is singular, then there exists a codimension one cell $e \in \Cell_{3n-4}(P_n)$ so that $\phi_v \in \bar{e}$. Further, suppose $e' \in \Cell_\bullet(P_n)$ and $\phi_v \in e'$. If $\phi_v$ is singular, then so is every $\phi_{v'} \in \bar{e'}$.
\end{lem}
\begin{proof}
If $\phi_v$ is singular, by definition it must have a pair of distant pipes which intersect. As noted in Remark \ref{rmk:overrun}, we can assume that these are, in fact, non-parallel pipes. It is easy to characterize plumbers' curves for which a given pair of non-parallel pipes intersect: (here and elsewhere, we write ``$(p - r)(q - r) \geq 0$" in place of ``neither $q < r < p$ nor $q < r < p$")\\
$\x_i \cap \y_j$ if 
\begin{enumerate}
\item $(v_i^x - v_{j+1}^x)(v_{i+1}^x - v_{j+1}^x) \geq 0$,
\item $(v_j^y-v_{i}^y)(v_{j+1}^y-v_{i}^y) \geq 0$ and
\item $v_i^z = v_j^z$,
\end{enumerate}
$\x_i \cap \z_j$ if 
\begin{enumerate}
\item $(v_i^x - v_{j+1}^x)(v_{i+1}^x - v_{j+1}^x) \geq 0$,
\item $v_i^y = v_{j+1}^y$ and
\item $(v_j^z-v_{i}^z)(v_{j+1}^z-v_{i}^z) \geq 0$,
\end{enumerate}
and $\y_j \cap \z_i$ if
\begin{enumerate}
\item $v_{i+1}^x = v_{j+1}^x$,
\item $(v_j^y - v_{i+1}^y)(v_{j+1}^y - v_{i+1}^y) \geq 0$ and
\item $(v_i^z-v_{j}^z)(v_{i+1}^z-v_{j}^z) \geq 0$.
\end{enumerate}
Each of these conditions requires some coordinate equality. However, the top dimensional cells in $\Cell_{3n-3}(P_n)$ contain only points for which there are no coordinate equalities, so the point $v$ must be located in some cell of codimension at least one, and thus must be in the closure of some cell of codimension one. 

Now, suppose $\phi_v$ is singular, so there is an intersection of non-parallel distant pipes in $\phi_v$ and at least one of the above collection of conditions is satisfied by $\phi_v$. Thus $\phi_v$ is contained in the closure of a cell $e'\in \Cell_{3n-4}(P_n)$ whose elements have only the single specified coordinate equality and the necessary coordinate inequalities to induce this intersection. As all elements in the interior $e'$ share the same pattern of coordinate equalities and inequalities, and the boundary of $e'$ consists of points with strictly more coordinate equalities (which cannot cause this intersection condition to fail), for all points $\phi_{v'} \in \bar{e}$, $\phi_{v'}$ also has this intersection of non-parallel distant pipes and thus is singular.
\end{proof}

Lemma \ref{lem:singinface} tells us both that $K_n$ is an open submanifold of $P_n$ and that the discriminant $S_n$ is described by a closed subcomplex of $P_n$ generated by codimension one cells, whose combinatorics we now carefully develop.

\begin{defn}
Let $\Cell_\bullet(S_n)$ be the cell complex on $S_n$ induced by all of the elements of $\Cell_{3n-4}(P_n)$ whose constituent plumbers' curves are singular.
\end{defn}

We require a naming convention for cells in these complexes involving coordinate equalities.

\begin{defn} Let $\sigma \in \Sigma_n$. We say $\tau\in \Sigma_n$ is a \emph{transposition appearing in $\sigma$} if it is of the form $\tau = (\sigma(i)\; \sigma(i+1))$ for some $i\in\{1\dots n-1\}$.
\end{defn}

For example, the transpositions appearing in 3124 are $(3 \; 1)$, $(1 \; 2)$ and $(2 \; 4)$. 

The elements of $\Cell_\bullet(P_n)$ carry left actions by the group $\Sigma_{n-1}$ for each coordinates, and we write $\rho_x \cdot e(\sigma_x, \sigma_y, \sigma_z) = e(\rho_x\sigma_x, \sigma_y, \sigma_z)$, extending to the other coordinates similarly.

Observe that precomposing $\sigma$ by $\tau$ exchanges two adjacent elements in permutation notation precisely when $\tau$ appears in $\sigma$. Such a transposition reverses a single coordinate inequality in the definition of a cell, so the transpositions appearing in the permutations $\sigma_x, \sigma_y$ and $\sigma_z$ thus enumerate the codimension one boundaries of the cell $e(\sigma_x, \sigma_y, \sigma_z)$. 

\begin{prop} Let $e(\sigma_x, \sigma_y, \sigma_z)\in \Cell_{3n-3}(P_n)$ and $\tau$ a transposition. Then $\bar{e} \;\cap\; \overline{\tau \cdot e}$ is of codimension 1 if and only if $\tau$ appears in $\sigma_x$. 
\end{prop}

This motivates the following naming convention for codimension 1 cells. 

\begin{defn} \label{def:codim1cell} Let $e(\sigma_x, \sigma_y, \sigma_z) \subseteq P_n$ be a cell and $\tau_x=(\sigma_x(i)\;\sigma_x(i+1))$ a transposition appearing in $\sigma_x$. Define a codimension one cell named $e(\sigma_x, \sigma_y, \sigma_z; \tau_x)$ using the same inequalities in equation (\ref{eqn:respectsigma}) as for $e(\sigma_x, \sigma_y, \sigma_z)$, but changing the inequality $v^x_{\sigma_x(i)} < v^x_{\sigma_x(i+1)}$ to an equality. 
\end{defn} 

For example, all plumbers curves $\phi_v$ contained in $e(1324_x, \sigma_y, \sigma_z; (2\_3)_x)$ satisfy $0 < v_1^x < v_3^x = v_2^x < v_4^x < 1$. 

As is clear from the example, however, this notation is redundant: each cell is named twice. However, the convention makes incidence relations with the $(3n-3)$-cells easy to understand and identification of duplicate names is straightforward.

\begin{prop} $e(\sigma_x, \sigma_y, \sigma_z; \tau_x) = e(\sigma_x', \sigma_y, \sigma_z; \tau_x)$ if and only if $\tau_x\sigma_x = \sigma_x'$. 
\end{prop}

We now extend this notation to name cells of any codimension. 

\begin{defn} \label{def:perminsigma} Let $\ell\co\Sigma_n \to \mathbb{N}$ be the standard Coxeter group length function, and $\sigma, \rho\in \Sigma_n$. Say $\rho$ \emph{is a permutation appearing in $\sigma$} if $\rho$ admits a decomposition into $\ell(\rho)$ transpositions, each of which is a transposition appearing in $\sigma$.
\end{defn}

If $\rho$ appears in $\sigma$, then its decomposition into transpositions appearing in $\sigma$ is unique up to reordering. This occurs because in a decomposition of $\rho$ into a collection of disjoint cycles $\{\rho^j\}$, $\sigma(i)$ and $\sigma(i+1)$ occur in the same $\rho^j$ if and only if $(\sigma(i) \;\sigma(i+1))$ is a transposition in the decomposition of $\rho$.

\begin{defn} \label{def:codimkcell} Let $e(\sigma_x, \sigma_y, \sigma_z) \in \Cell_{3n-3}(P_n)$ and $\rho_x, \rho_y, \rho_z \in \Sigma_{n-1}$ be permutations such that $\rho_i$ appears in $\sigma_i$, $i\in\{x,y,z\}$. Define the cell $e(\sigma_x, \sigma_y, \sigma_z; \rho_x, \rho_y, \rho_z)\in \Cell_{3n-3-(\ell(\rho_x) + \ell(\rho_y) + \ell(\rho_z))}(P_n)$ by, for each transposition occuring in the chosen decomposition of the $\rho_i$, $i\in\{x,y,z\}$, changing the appropriate inequality in equation (\ref{eqn:respectsigma}) to an equality (as in Definition \ref{def:codim1cell}).
\end{defn}

The uniqueness of the decomposition of each $\rho_i$ shows that this cell naming convention is well defined, as there is only one choice of equalities associated to a given cell name. We can determine how many different names apply to a given cell from the structure of the $\rho_i$.

\begin{prop} Let $e(\sigma_x, \sigma_y, \sigma_z; \rho_x, \rho_y, \rho_z) \in \Cell_\bullet(P_n)$. Decompose each $\rho_i$, $i\in\{x, y, z\}$, into a product of $k(i)$ of disjoint cycles, $\rho_i^1\rho_i^2\cdots\rho_j^{k(i)}$. The cell $e$ has
\begin{eqnarray*}
\#e&=&\prod_{i\in\{x, y, z\}}\left[\prod_{j=1}^{k(i)}(\ell(\rho_i^j)+1)!\right]
\end{eqnarray*}
redundant labels. Each such label uniquely identifies a top dimensional cell in whose closure $e$ occurs.
\end{prop}

\begin{proof}
Let $\rho_x^j$ be a cycle in the decomposition of $\rho_x$ into disjoint cycles. The requirement that $\rho_x^j$ appear in $\sigma_x$ tells us that this subset of $\{1,\dots,n-1\}$ on which $\rho_x^j$ has nontrivial action is precisely $\{\sigma(t), \sigma(t+1),\dots,\sigma(t+\ell(\rho_x^j))\}$ for some $t$. Now, the cell $e(\sigma_x, \sigma_y, \sigma_z; \rho_x^j)$ is defined to be all maps $v$ respecting $\sigma_y$ in $y$, $\sigma_z$ in $z$ and so that
\begin{eqnarray*}
\cdots <  v^x_{\sigma_i(t-1)} < v^x_{\sigma_i(t)} = v^x_{\sigma_i(t+1)} = \cdots = v^x_{\sigma_i(t+\ell(\rho_x^j))} < v^x_{\sigma_i(t+\ell(\rho_x^j)+1)} \cdots 
\end{eqnarray*}
That is, decomposing $\rho_x$ into disjoint cycles is equivalent to isolating blocks of adjacent of equalities in equation (\ref{eqn:respectsigma}). 

To select a cell $e(\sigma_x', \sigma_y, \sigma_z)$ which has in its closure $e(\sigma_x, \sigma_y, \sigma_z, \rho_x^i)$, we choose an order for these vertices. There are $(\ell(\rho_x^j)+1)!$ such, and it is clear that this choice is independent of the corresponding choices for the other disjoint cycles in $\rho_x$. 

As we can perform this procedure independently for each of the disjoint cycles and for each coordinate, our final count is simply the product of that for the disjoint cycles, as required.
\end{proof}

For example, there are two classes of names for codimension 2 cells: those for which the $\rho$ are some combination of two disjoint transpositions and those for which the $\rho$ consist of a single 3-cycle. In the former case, there are 4 names for each cell, while in the latter there are 6. 

Utilizing this naming scheme, we can extend the idea of representative knots for top dimensional cells to produce representative curves for every cell in $\Cell_\bullet(P_n)$. 

\begin{defn} \label{def:repcurve} Fix a cell $e = e(\sigma_x, \sigma_y, \sigma_z; \rho_x, \rho_y, \rho_z) \subseteq P_n$. A map in such a cell will have vertices in $n - \ell(\rho_x)$ distinct ($y$-$z$)-planes, some of the $x$ vertices now being coplanar. Define $E[\sigma_x, \rho_x]: \{1,\dots, n-1\} \to \mathbb{N}$ to take as its value the number of equalities that occur before $v_i$ in the $x-axis$ for elements of $e$.



The \emph{representative curve} (or \emph{representative knot}, if non-singular) for $e$ is $\phi_{v(e)}$ with 
\begin{displaymath}
v(e)_i=\left(\frac{E[\sigma_x, \rho_x](i)}{n-\ell(\rho_x)}, \frac{E[\sigma_y, \rho_y](i)}{n-\ell(\rho_y)}, \frac{E[\sigma_z, \rho_z](i)}{n-\ell(\rho_z)}\right).
\end{displaymath}
\end{defn}

\subsection{An algorithm for computing components of $K_n$}

Our naming convention for cells allows us to resolve several fundamental questions about the spaces of plumbers' knots algorithmically.

By Lemma \ref{lem:singinface}, the codimension 1 cells in $C_{3n-4}(P_n)$ are partitioned into those whose elements are plumbers' knots and which consist of singular plumbers' curves. As those which consist of singular plumbers' curves generate a cell complex for $S_n$, we wish to distinguish them. Fortunately, our naming convention tells us into which of these classes a given cell falls.

Choose the representative knot for a cell $e(\sigma_x, \sigma_y, \sigma_z)$ and a transposition $\tau = (a\; b)$ appearing in $\sigma_x$. At the level of plumbers' curves, the left action of $\tau$ on $\sigma_x$ corresponds to exchanging the x-coordinates of the $a$th and $b$th vertices of the knots. A corresponding path in $P_n$ between the representative knots passes through the codimension one cell $e(\sigma_x, \sigma_y, \sigma_z; \tau_x)$, and thus must involve at least one curve where the ($y$-$z$)-planes containing the pipes $\y_{a-1}$ and $\z_{a-1}$ is the same as that containing $\z_{b-1}$ and $\y_{b-1}$. These pipes may intersect in various combinations, as illustrated in Figure \ref{fig:xswap}, and any other intersections will only occur in the closure of the cell, as discussed in Remark \ref{rmk:overrun}. 

Similar considerations demonstrate that transpositions appearing in $\sigma_y$ (respectively, $\sigma_z$) can only cause intersections of the form $\x_a \cap \z_{b-1}$ or $\x_{b} \cap \z_{a-1}$ (respectively, $\x_a \cap \y_b$ or $\x_b \cap \y_a$).

\begin{figure} 
\begin{center}
\psfrag{C1}{$x = v_{a}^x$}
\psfrag{C2}{$x = v_{b}^x$}
\psfrag{C3}{$y_{a-1}$}
\psfrag{C4}{$z_{a-1}$}
\psfrag{C5}{$y_{b-1}$}
\psfrag{C6}{$z_{b-1}$}
\psfrag{C7}{$x = v_{a}^x = v_{b}^x$}
\psfrag{C8}{$y_{a-1}\cap z_{b-1}$}
\psfrag{C9}{$y_{b-1} \cap z_{a-1}$}
\includegraphics[width=9cm]{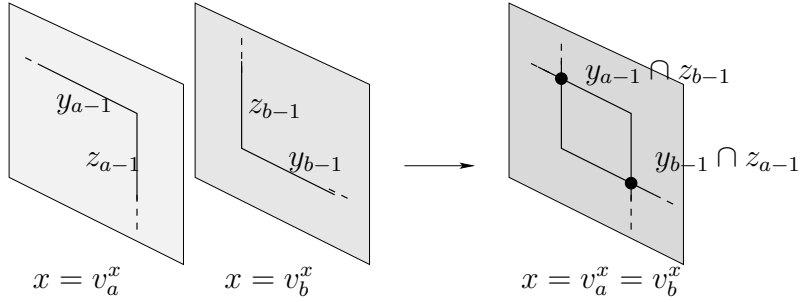}
\caption{Transposing x-coordinates may result in y-z intersections.}
\label{fig:xswap}
\end{center}
\end{figure}

\begin{defn} \label{def:produceint} Let $e(\sigma_x, \sigma_y, \sigma_z) \in \Cell_{3n-3}(P_n)$ and $\tau_x = (a\;b)$ a transposition appearing in $\sigma_x$. We say that \emph{$\tau_x$ produces an intersection for $e$} if one of the following pairs of conditions holds.

Either 
\begin{enumerate}
\item $\sigma_y^{-1}(b)$ is between $\sigma_y^{-1}(a - 1)$ and $\sigma_y^{-1}(a)$; and $\sigma_z^{-1}(a - 1)$ is between $\sigma_z^{-1}(b - 1)$ and $\sigma_z^{-1}(b)$; or
\item $\sigma_y^{-1}(a)$ is between $\sigma_y^{-1}(b - 1)$ and $\sigma_y^{-1}(b)$; and $\sigma_z^{-1}(b-1)$ is between $\sigma_z^{-1}(a - 1)$ and $\sigma_z^{-1}(a)$
\end{enumerate}
\end{defn}

These conditions translate into those for $\y_{a-1} \cap \z_{b-1}$ and $\y_{b-1}\cap\z_{a-1}$ in the proof of Lemma \ref{lem:singinface} at the level of coordinates. The definitions for producing intersections in the $y$- and $z$-coordinates are similar.

\begin{thm} \label{thm:isotopyint} Let $e$ and $\tau_x$ be as in Definition \ref{def:produceint}. There is a straight-line geometric isotopy between the representative knots for $e$ and $\tau_x\cdot e$ if and only if $\tau_x$ does not produce a y-z intersection for $e$. 
\end{thm}

\begin{defn}
Let $e$ and $\tau_x$ be as in Definition \ref{def:produceint}. If an isotopy from $\phi_{v(e)}$ to $\phi_{v(\tau \cdot e)}$ exists, it is an \emph{elementary geometric isotopy}.
\end{defn}

Elementary geometric isotopies play a role for plumbers' knots similar to that played by Reidermeister moves for diagrams of $C^1$ knots. However, the Reidermeister moves change various complexity measures of a knot diagram, while our notion of complexity (number of moves) is invariant under elementary geometric isotopy. This allows us to consider the equivalence relation generated by such moves in the context of specific spaces $K_n$.

\begin{cor} Let $\phi_{w_1}\in e_1$, $\phi_{w_2}\in e_2$. Then $\phi_{w_1}$ is geometrically isotopic to $\phi_{w_2}$ if and only if there is a sequence of elementary geometric isotopies connecting $\phi_{v(e_1)}$ to $\phi_{v(e_2)}$.
\end{cor}

That is,

\begin{cor} \label{cor:h0freeab} $H_0(K_n) \cong \mathbb{Z}\langle e(\sigma_x, \sigma_y, \sigma_z)\;|\;\sigma_x, \sigma_y, \sigma_z \in \Sigma_{n-1}\rangle/\sim$, where $\sim$ is the equivalence relation generated by elementary geometric isotopies.
\end{cor}

Indeed, these observations reduce the problem of computing isotopy classes of $n$-move plumbers knots to that of computing components of a graph.

\begin{defn}
Let $\Gamma_n = (V_n, E_n)$ by the graph with $V = \Cell_{3n-3}(P_n)$ and $(e, e') \in E$ if there is an elementary genometric isotopy between the representative knots for $e$ and $e'$.
\end{defn}

The results of computing the components of this graph for $K_5$ and $K_6$ are included in Table \ref{tab:knottypes}. The topological isotopy classes of these knots were determined by computation of the Alexander polynomial on representatives of each class. The code for each of these processes is appended to the end of the \TeX\ file for this paper on the arXiv.
 
\begin{table}

\begin{tabular}{c|c|c|c|c}
\hline
\multicolumn{5}{c}{Components of $K_5$}\\
\hline
Type&Cells&Representative&Type&Representative\\
\hline
$0_1$ & 13,728 & $1234_x, 1234_y, 1234_z$ \\ \hline
\multirow{3}{*}{$3_1$\! (R)} &31& $1342_x,2413_y,2413_z$& \multirow{3}{*}{$3_1$\! (L)} & $1342_x,3142_y,2413_z$\\ 
& 16 & $1342_x,2413_y,3124_z$&& $1342_x,3142_y,3124_z$\\ 
& 1 & $2431_x,2413_y,4213_z$&& $2431_x,3142_y,4213_z$\\ \hline
\multicolumn{5}{c}{Components of $K_6$}\\
\hline
$0_1$ & 1.7m & $12345_x, 12345_y, 12345_z$ \\ \hline
\multirow{3}{*}{$3_1$\! (R)} & 19,507 & $24135_x,31245_y,23145_z$&\multirow{3}{*}{$3_1$\! (L)} & $12453_x,13524_y,13524_z$ \\ 
& 5 & $42351_x,24315_y,24135_z$&  &$42351_x,51342_y,24135_z$\\ 
& 5 & $13524_x, 15324_y, 51342_z$&&  $13524_x, 42351_y, 51342_z$\\
 \hline
\multirow{2}{*}{$4_1$}& 393 & $14352_x, 31452_y, 42135_z$&\multirow{2}{*}{$4_1$}&  $31452_x,31524_y,32451_z$\\
&393&  $24153_x,25314_y,24315_z$&&  $24513_x,42135_y,32415_z$\\
\hline
\multirow{5}{*}{$5_1$\! (R)}& 19&  $24153_x,31524_y,42315_z$&\multirow{5}{*}{$5_1$\! (L)}&  $15342_x,31542_y,31524_z$\\
&19& $25134_x,41253_y,35241_z$&&  $52413_x,24513_y,25314_z$\\
&4&  $15342_x,24153_y,42153_z$ &&   $15342_x,31542_y,42513_z$\\
&4& $31542_x,31524_y,42315_z$&&  $31542_x,42513_y,42315_z$\\
&1&  $41523_x,41352_y,34152_z$&&  $41523_x,25314_y,34152_z$\\
\hline
\multirow{14}{*}{$5_2$\! (R)} &12& $15342_x,24513_y,35124_z$&\multirow{14}{*}{$5_2$\! (L)}&  $15342_x,31542_y,35124_z$\\
&12&  $25413_x,35124_y,25314_z$&&  $24513_x,42153_y,42315_z$\\
&9&  $25134_x,24153_y,35241_z$&&  $25134_x,35124_y,35241_z$\\
&9&  $25413_x,31524_y,42315_z$&&  $52413_x,24513_y,23514_z$\\
&4&  $15342_x,24513_y,42153_z$&&  $15342_x,31542_y,42153_z$\\
&4&  $31542_x,31524_y,42351_z$&&  $31542_x,42153_y,42315_z$\\
&3&  $15342_x,25413_y,31524_z$&&  $15342_x,31452_y,31524_z$\\
&3&  $24153_x,35214_y,42315_z$&&  $24153_x,41253_y,42315_z$\\
&2&  $15342_x,25413_y,42513_z$&&  $15342_x,31452_y,42513_z$\\
&2&  $35142_x,35214_y,42315_z$&&  $35142_x,41253_y,42315_z$\\
&1&  $15342_x,24153_y,41253_z$&&  $15342_x,35142_y,41253_z$\\
&1&  $31452_x,31524_y,42315_z$&&  $31452_x,42513_y,42315_z$\\
&1&  $41523_x,25314_y,43152_z$&&  $41523_x,41352_y,43152_z$\\
&1&  $41532_x,41352_y,34152_z$&&  $41532_x,25314_y,34152_z$\\
\hline
\end{tabular}

\caption{Components of $K_5$ and $K_6$; the number of cells in a component are the same in the second column}
\label{tab:knottypes}
\end{table}

We can also clearly computationally determine if two n-move plumbers' knots $\phi_{w_1}$ and $\phi_{w_2}$ are geometrically isotopic. Let $e_1, e_2\in V_n$ with $\phi_{w_1}\in e_1$, $\phi_{w_2} \in e_2$. Starting at $e_1$, perform a search for $e_2$. If the search terminates successfully, the knots are isotopic. Otherwise, they are not geometrically isotopic for the given n. 

Suppose $\varphi_1$ and $\varphi_2$ are $C^1$ knots. If it is possible to determine an $n$ so that approximating each via a plumbers' knot of $n$-moves ensures that a geometric isotopy exists if topological isotopy exists, this method can be used to determine if given knots are isotopic in $O((n!)^3)$ running time.

An immediate consequence of Theorem \ref{thm:isotopyint} and Lemma \ref{lem:singinface}, we have the following characterization of the cell complex for $S_n$.

\begin{cor} \label{cor:codim1cells} $\Cell_\bullet(S_n)$ is generated in dimension $3n-4$ as a cell complex by all cells of the form $e(\sigma_x, \sigma_y, \sigma_z, \tau_x)$ for which $\tau_x$ produces an intersection. 
\end{cor}

Equivalent to constructing the graph $\Gamma_n$ is the task of enumerating the cells in $\Cell_{3n-4}(S_n)$, as these cells correspond to edges not appearing in $E_n$.

\section{The map $\iota_n: P_n \to P_{n+1}$}\label{sec:maps}

For fixed $n$, elements of $P_n$ are too rigid to properly model classes of $C^1$ knots, as we see from the multiple representatives of various topological isotopy classes in Table \ref{tab:knottypes}. Therefore, we require a mechanism by which to increase the articulation of a knot of interest in a fashion which varies continuously across each $P_n$ and which, in the limit, produces a model of the space of $C^1$ knots. It will be convenient to construct the maps $P_n \to P_{n+1}$ so that they take $S_n$ to $S_{n+1}$ in addition to carrying $K_n$ to $K_{n+1}$. Such maps clearly involve adding vertices to curves, a process which is complicated by the presence of zero-length pipes in non-singular curves.

While it would also be useful for the map to be cellular with respect to the cell complex structure on $P_n$, and we will not be able to construct such a map. However, we can approximate a cellular map by applying a subdivision to the cell complex in the codomain, after which the map will be ``mostly cellular'' in the sense we can define a cellular map from a large subset of $\Cell_\bullet(P_n)$ to the subdivision of $\Cell_\bullet(P_{n+1})$. The remainder of the curves of $P_n$ will be forced to map into the interior of cells in $P_{n+1}$, but we will be able to identify which cells contain their images. 

Occasionally it will be useful to refer to our basis vectors $\{x, y, z\}$ as $\{e_0, e_1, e_2\}$. When we write $[n]_3$ we are referring to the reduction of $n$ modulo 3. 

We require two more pieces of notation in order to understand how the maps in the directed system interact with the cell complex. First, we need to modify permutations to reflect the insertion of vertices into the plumbers' maps.

\begin{defn} \label{def:jmaps}
Let $k\in\{1\dots n-1\}$ and $\sigma \in \Sigma_{n-1}$. Define  $j_{k, k+1}(\sigma)\in\Sigma_{n}$ by

\begin{displaymath}
j_{k, k+1}(\sigma)(i)=
\begin{cases}
\hat{\sigma}[k](i)&i < \sigma^{-1}(a) + \lfloor \frac{1}{2}(\sigma^{-1}(k) + \sigma^{-1}(k+1)\rfloor)\\
k+1&i = \sigma^{-1}(a) + \lfloor \frac{1}{2}(\sigma^{-1}(k) + \sigma^{-1}(k+1)\rfloor)\\
\hat{\sigma}[k](i-1)&i > \sigma^{-1}(a) + \lfloor \frac{1}{2}(\sigma^{-1}(k) + \sigma^{-1}(k+1)\rfloor).
\end{cases}
\end{displaymath}
\noindent where $\hat{\sigma}[k] : [n-1] \to [n]$ by
\begin{displaymath}
\hat{\sigma}[k](i)=
\begin{cases}
\sigma(i)& \sigma(i) \leq k\\
\sigma(i)+1& \sigma(i) > k.
\end{cases}
\end{displaymath}
\noindent and $a = \min\{\sigma^{-1}(k), \sigma^{-1}(k+1)\}$.
\end{defn}

The permutation $j_{k,k+1}(\sigma)$ is the result of inserting a new vertex at the ``lexicographic midpoint'' of $k$ and $k+1$ in permutation notation. For example, $j_{2,3}(12453) = 125364$. 


Second, we need to define a subdivision of the complex $\Cell_\bullet(P_n)$ and extend our previous naming convention to its cells. Recall that elements of $\Cell_\bullet(P_n)$ are geometrically defined to be the product of three simplices. 

\begin{defn}
Let $\Cell_\bullet^{\;B}(P_n)$ be the complex obtained by barycentric subdivision of elements of $\Cell_\bullet(P_n)$ separately in all three coordinates. 
\end{defn}

In the simplex in $\I^{(n-1)}$ which respects $\sigma_x$, the codimension one subset in which $v_j^x$ is the average of its two neighboring coordinates is a union of faces in the barycentric subdivision, as in Figure \ref{fig:bary}. We will require the following notation for such elements in $\Cell_\bullet^{\;B}(P_n)$.

\begin{eqnarray*}
e(\sigma_x, \sigma_y, \sigma_z; \langle j\rangle_x)&=&\{v \in e(\sigma_x, \sigma_y, \sigma_z) |\\&&\quad v_j^x = \frac{1}{2} (v_{\sigma_x(\sigma_x^{-1}(j)-1)}^x + v_{\sigma_x(\sigma_x^{-1}(j)+1)}^x)\}
\end{eqnarray*}

The cell $e(\sigma_x, \sigma_y, \sigma_z; \langle j\rangle_x)$ is a product of the simplicies represented by $\sigma_y$ and $\sigma_z$ with these subdivision faces. Since this cell is a subset of an existing cell, we extend our transposition notation to its faces in the obvious manner.

\begin{figure} 
\begin{center}
\psfrag{t1}{$t_1$}
\psfrag{t2}{$t_2$}
\psfrag{C1}{$t_2 = \frac{1}{2}(t_1 + t_3)$}
\psfrag{C2}{$t_1 = \frac{1}{2}(t_0 + t_2)$}
\includegraphics[width=6cm]{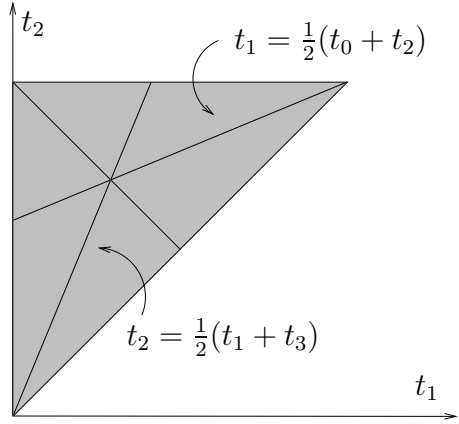}
\caption{Averages of coordinates in barycentric subdivision; here, $t_0 = 0$ and $t_3 = 1$}
\label{fig:bary}
\end{center}
\end{figure}

We now wish to define the maps $\iota_n : P_n \to P_{n+1}$. For generic curves in $P_n$, such a map is simply a selection of a pipe into which to insert a new vertex. In order for the limit of the resulting directed system to provide a model for the space of knots, these maps should insert new vertices in a ``sufficiently distributed" fashion as $n$ varies. One approach to ensuring such a distribution, which we use here to avoid the complexity of an explicit formula, is to fix a function $A\co\N \to \N$ by sampling $\alpha(n)$ from the uniform distribution on $\{1, \dots, n-1\}$ and insert a vertex at the midpoint of the pipe traveling in the $e_{[n]_3}$ direction in the $A(n)$th move, per Figure \ref{fig:iotagt}. We now define the maps on both the plumbers' map and cellular level in parallel. 

For our fixed $A$ and $n$, most plumbers' knots and all singular plumbers' curves can be articulated simply by adding a new vertex in the image of the existing knot in the specified pipe.

\begin{defn} \label{def:good}
Fix a function $A: \mathbb{N} \to \mathbb{N}$ as described above, and let $n \in \mathbb{N}$. Write $\alpha = A(n)+1$.
Let $\phi_v\in P_n$ so that there is some $i$ with $(v_\alpha^{[n]_3} - v_i^{[n]_3})(v_{\alpha+1}^{[n]_3} - v_i^{[n]_3}) > 0$. Such points live in cells for which $(\alpha\;\alpha\!+\!1)$ does not appear in $\sigma_{e_{[n]_3}}$. Call any cell in with this property \emph{good} and write $C_\bullet^{\;g}(P_n)$ for the collection of all such.
\end{defn}

We begin by defining the map on good cells. For technical reasons, we assume $n>2$; this restriction is trivial as $P_1$ and $P_2$ are uninteresting.

\begin{defn} \label{def:iotag} Let $A$, $n>2$ and $\alpha$ be as in Definition \ref{def:good}. Let $e\in \Cell_\bullet^{\;g}(P_n)$ and $\phi_v \in e$ and write $\delta_{i,j}$ for the Kronecker delta. Define  $\hat{\iota}^{g}_n(v)$ by
\begin{eqnarray*}(\iota^{g}_n(v))_i&=&\left\{\begin{array}{ll}
v_i&i \leq \alpha\\
v_{i-1}&i > \alpha+1
\end{array}\right.
\end{eqnarray*}
and
\begin{eqnarray*}
(\hat{\iota}^{g}_n(v))_{\alpha+1}^x&=& v_{\alpha+1}^x - \tfrac{1}{2}s(\x_\alpha)||\x_\alpha||\delta_{{[n]_3},0}\\
(\hat{\iota}^{g}_n(v))_{\alpha+1}^y&=&v_{\alpha}^y + \tfrac{1}{2}s(\y_\alpha)||\y_\alpha||[n-1]_3\\
(\hat{\iota}^{g}_n(v))_{\alpha+1}^z&=&v_{\alpha}^z + \tfrac{1}{2}s(\z_\alpha)||\z_\alpha||\delta_{{[n]}_3,2}. \end{eqnarray*}
and define $\iota^g_n(\phi_v) = \phi_{\hat{\iota}^g_n}$.
\end{defn}

The image of the curve $\iota^g_n(\phi_v)$ is the same as that of $\phi_v$, but it contains a new vertex which is coplanar in two dimesions with vertices from the old curve and whose third coordinate is the average of the two endpoints of the pipe. This is geometrically isotopic to a curve in which the third coordinate is inserted instead at the average of those of the two vertices closest to the middle of the pipe. 

Using the $j$ maps of Definition \ref{def:jmaps}, the following is immediate from Definition \ref{def:iotag}.

\begin{lem} Let $e\in\Cell^{\;g}_\bullet(P_n)$. Then the image of $e$ under the induced chain map $(\iota_n^{g})_\# \co \Cell_\bullet^{\;g}(P_n) \to \Cell_\bullet^{\;B}(P_n)$ is

\begin{eqnarray*}
(\iota^{g}_n)_\#(e(\sigma_x, \sigma_y, \sigma_z)) &=& e(j_{\alpha, \alpha+1}\sigma_x, j_{\alpha}\sigma_y, j_\alpha\sigma_z;\\
&& \quad\langle \alpha\!+\!1\rangle_x, (\alpha\; \alpha\!+\!1)_y, (\alpha\; \alpha\!+\!1)_z)..
\end{eqnarray*}

\end{lem}

\begin{figure} 
\begin{center}
\psfrag{C1}{$v_{\alpha}$}
\psfrag{C2}{$v_{\alpha+1}$}
\psfrag{C4}{$\hat{\iota}^{g}_n(v)_{\alpha}$}
\psfrag{C3}{$\hat{\iota}^{g}_n(v)_{\alpha+1}$}
\psfrag{C5}{$\hat{\iota}^{g}_n(v)_{\alpha+2}$}
\includegraphics[height=3cm]{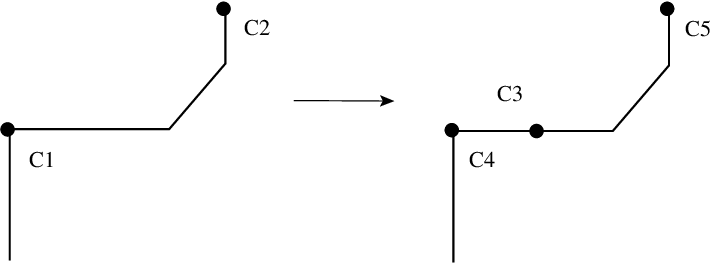} 
\caption{$\iota^{g}_n$}
\label{fig:iotagt}
\end{center}
\end{figure}

However, if the pipe into which we are to insert our prescribed vertex has zero length, such a map will insert a vertex coincident to an existing one. This poses no problem in the discriminant, but would result in a knot mapping to a singular curve. For example, if $v_i = (v_i^x, v_i^y, v_i^z)$, and $v_{i+1} = (v_i^x, v_i^y, v_{i+1}^z)$, attempting to apply $\iota^{g}_n$ to insert a vertex in either of $\x_i$ or $\y_i$ would result in $\hat{\iota}^{g}_n(v)_i = \hat{\iota}^{g}_n(v)_{i+1} = v_i$, so $\iota^{g}_n(\phi_v)\left(\frac{3i+1}{n}\right) = \iota^{g}_n(\phi_v)\left(\frac{3i+4}{n}\right)$, an intersection of distant pipes. Notice that this issue only arises if a particular pair of consecutive vertices specified by $A$ and $n$ in the curve lie in the same coordinate plane. 

\begin{defn}
Let $A$, $n$ and $\alpha$ be as in Definition \ref{def:good}. A cell of the form \mbox{$e(\sigma_x, \sigma_y, \sigma_z; (\alpha \; \alpha\!+\!1)_{e_{[n]_3}}) \in \Cell_\bullet(P_n)$} is called \emph{$\alpha$-planar}. Write $\Cell^{\;p}_\bullet(P_n)$ for the family of all such cells.
\end{defn}

We will resolve this issue by ``borrowing'' length from the two pipes which neighbor the zero-length pipe into which we are to insert the vertex, at least one of which is guaranteed to have non-zero length if the curve is non-singular.  In some cases, this requires changing the coordinates of the two existing neighboring vertices in a manner which does not change the topological knot type of the image, as in Figure \ref{fig:iotaa}. Thus, we limit our deformation of the knot in either borrowing direction to half the distance to the closest pipe, which must run perpendicular to the plane in which the two vertices occur.

We make the following sequence of definitions for $e_{[n]_3} = x$. Similar definitions are assumed for $y$ and $z$.

\begin{defn} \label{def:iotaa} Let $A$, $n$ and $\alpha$ be as in Definition \ref{def:good}. Let $e \in \Cell_\bullet^{\;p}(P_n)$ and $\phi_v \in e$. Define $\hat{\iota}^p_n(v)$ by
\begin{eqnarray*}(\hat{\iota}^p_n(v))_i&=&\left\{\begin{array}{ll}
v_i&i \leq \alpha-1\\
(v_{\alpha}^x, v_{\alpha}^y, \tfrac{1}{2}(v_{\alpha}^z + v^z_{\sigma_z (\sigma_z^{-1}(\alpha) + s(\z_{\alpha-1}))}))& i = \alpha\\
(v_{\alpha}^x, \tfrac{1}{2}(v_{\alpha}^y + v^y_{\sigma_y (\sigma_y^{-1}(\alpha) + s(\y_{\alpha}))}), v_\alpha^z) &i = \alpha+1\\
v_{i-1}&i \geq \alpha+2.
\end{array}\right.
\end{eqnarray*}
and $\iota^p_n(\phi_v) = \phi_{\hat{\iota}_n^p(v)}$.
\end{defn}

\begin{figure} 
\begin{center}
\psfrag{C1}{$v_{\alpha}$}
\psfrag{C2}{$v_{\alpha+1}$}
\psfrag{C4}{$\hat{\iota}_n^p(v)_{\alpha}$}
\psfrag{C3}{$\hat{\iota}_n^p(v)_{\alpha+1}$}
\psfrag{C5}{$\hat{\iota}_n^p(v)_{\alpha+2}$}
\includegraphics[height=3cm]{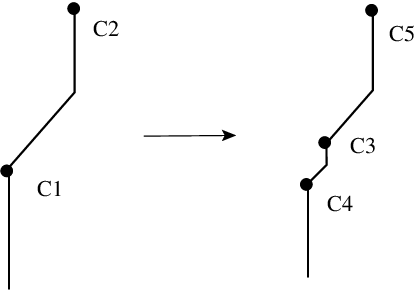} 
\caption{$\iota^P_n$}
\label{fig:iotaa}
\end{center}
\end{figure}

As before, this map descends to a cellular map to the barycentric subdivision of the codomain.

\begin{lem} Let $e\in\Cell_\bullet^{\;p}(P_n)$. Then the image of $e$ under the induced chain map $(\iota_n^p)_\#\co \Cell_\bullet^{\;p}(P_n) \to \Cell_\bullet^{\;B}(P_n)$ is

\begin{eqnarray*}
(\iota_n^P)_\#(e(\sigma_x, \sigma_y, \sigma_x; (\alpha\; \alpha\!\!+\!\!1)_x)) &=&e(j_{\alpha, \alpha+1}\sigma_x, j_\alpha\sigma_y, j_\alpha\sigma_z;\\
&& \quad (\alpha\;\alpha\!+\!1)_x\langle \alpha\!+\!1\rangle_x, \langle\alpha\!+\!1\rangle_y, \langle\alpha\!+\!1\rangle_z).
\end{eqnarray*}
\end{lem}

It now remains to continuously define a map on the remaining cells that continuously extends $\iota^{g}_n$ and $\iota^{p}_n$..

\begin{defn}
All cells which are neither good nor $\alpha$-planar are \emph{interpolating}. Write $C_\bullet^{\;I}(P_n)$ for the set of all interpolating cells.
\end{defn}

Interpolating cells are those for which $(\alpha \;\alpha+1)$ appears in $\sigma_{[n]_3}$ but which are not in the closure of an $\alpha$-planar cell. When within such a cell, rather than simply inserting a vertex we will also slightly perturb the vertices using the borrowing construction from Definition \ref{def:iotaa} in a fashion which continuously interpolates between the maps on good cells and those on $\alpha$-planar ones. See Figure \ref{fig:iotaI} for an example of such an interpolating map, comparing with Figure \ref{fig:iotagt} and Figure \ref{fig:iotaa}.

\begin{defn} \label{def:iotaI} Let $A$, $n$ and $\alpha$ be as in Definition \ref{def:iotag}, $e\in\Cell_\bullet^{\;I}(P_n)$ and $\phi_v \in e$. 

The \emph{interpolating parameter} for $v$, $p(v)$, is the ratio of the distance between $v_\alpha^x$ and $v_{\alpha+1}^x$ to the maximal distance between this pair of vertices in the collection of maps within the cell $e$ for which all other vertices agree with those of $\phi_v$:
\begin{eqnarray*}
p(v)&=& \frac{||\x_\alpha||}{|v^x_{\sigma_x(\sigma_x^{-1}(\alpha) - s(\x_\alpha)))}-v^x_{\sigma_x(\sigma_x^{-1}(\alpha+1) + s(\x_\alpha)))}|}
\end{eqnarray*}
\noindent Define $\iota^I_n(v)$ by
\begin{eqnarray*}(\iota^0_n(v))_i&=&\left\{\begin{array}{ll}
v_i&i \leq \alpha-1\\
(v_{\alpha}^x, v_{\alpha}^y, v_{\alpha}^z - \tfrac{p(v)}{2}(v_{\alpha}^z - v^z_{\sigma_z (\sigma_z^{-1}(\alpha) + s(\z_{\alpha-1}))}))& i = \alpha\\
(v_{\alpha}^x, v_{\alpha}^y - \tfrac{p(v)}{2}(v_{\alpha}^y - v^y_{\sigma_y (\sigma_y^{-1}(\alpha) + s(\y_{\alpha}))}), v_\alpha^z) &i = \alpha+1\\
v_{i-1}&i \geq \alpha+2.
\end{array}\right.
\end{eqnarray*}
\end{defn}

Only $\iota^I_n$ fails to be cellular in any sense, though the image of an interpolating cell is contained in a cell we can identify. 

\begin{lem} Let $\alpha = \alpha(n)+1$ and $e\in \Cell_\bullet^{\;I}(P_n)$. Thought of as a subset of $P_n$, the image of $e$ under $\iota^I_n$ satisfies
\begin{eqnarray*}
\iota^I_n(e(\sigma_x, \sigma_y, \sigma_x)) &\subseteq&e(j_{\alpha, \alpha+1}\sigma_x, j_\alpha\sigma_y, j_\alpha\sigma_z; \langle \alpha+1\rangle_x).
\end{eqnarray*}
\end{lem}
\begin{figure} 
\begin{center}
\psfrag{C1}{$v_{\alpha}$}
\psfrag{C2}{$v_{\alpha+1}$}
\psfrag{C4}{$\hat{\iota}^I_n(v)_{\alpha}$}
\psfrag{C3}{$\hat{\iota}^I_n(v)_{\alpha+1}$}
\psfrag{C5}{$\hat{\iota}^I_n(v)_{\alpha+2}$}
\includegraphics[height=3cm]{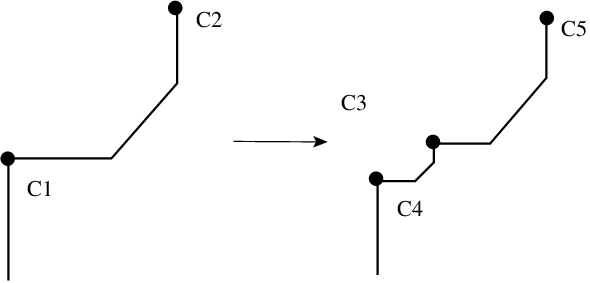} 
\caption{$\iota^I_n$}
\label{fig:iotaI}
\end{center}
\end{figure}

Finally, we can define the desired map $\iota_n : P_n \to P_{n+1}$.

\begin{defn} \label{def:iota}
Define $\iota_n$ to be the union of the maps $\iota_n^{\;g}, \iota_n^{\;p}$ and $\iota_n^{\;I}$. For $m > n$, write $\iota_{n, m} = \iota_{m-1} \circ \iota_{m-2} \circ \dots \circ \iota_n$.  Let $\iota_{n, \infty}\co K_n \to \dirlim K_n$ be the induced map and similarly for $S_n$.  By convention, $\iota_{n, n}$ is the identity. 
\end{defn}

By abuse, as it will be clear from context, we write $\iota_n$ for its restrictions to either $K_n$ and $S_n$.

Now, as we have been careful in our definitions to ensure that knots are taken to knots and singular curves to singular curves in an injective manner, we have

\begin{lem} \label{lem:dirsys} $\iota_n\co K_n\into K_{n+1}$ and $\iota_n\co S_n \into S_{n+1}$.
\end{lem}

\subsection{The directed system of spaces of plumbers knots}

The spaces $\{K_n\}_{n \in \mathbb{N}}$ form a directed system under the maps $\iota_n$ of Definition \ref{def:iota}. In the colimit, the curves are ``fully articulated'' and, as we will see in Theorem \ref{thm:htpytypecorrect}, the resulting space has the homotopy type of the space of $C^1$ long knots.

It follows from Lemma \ref{lem:dirsys} that if $\phi_v$ and $\phi_w\in K_n$ are geometrically isotopic, so are $\phi_{\iota_n(v)}$ and $\phi_{\iota_n(w)}$. 

\begin{defn}  Let $\phi_v \in K_n$ and $\phi_w\in K_m$. We say $\phi_v$ and $\phi_w$ are \emph{(topologically) isotopic} if there is a path $\bar\Phi_{v, w}\co \I \to \dirlim K_n$ with $\bar\Phi(0) = \iota_{n, \infty}(v)$ and $\bar\Phi(1) = \iota_{m, \infty}(w)$.
\end{defn}

In order to compare these new spaces $K_n$ to more familiar spaces of (long) knots, we will rely on the more familiar spaces of piecewise-linear knots. To facilitate comparison, we will restrict ourselves to a deformation retract of the usual space of such knots: PL maps of exactly $n$ segments such that each segment has as its domain an interval of length $\frac{1}{n}$. Such maps can, as with plumbers' curves, be uniquely identified with their vertices.

\begin{defn}   \label{def:plmap} Let $v\in ({\rm Int}\;\I^3)^{n-1}$, and for $i\in \{0,\dots, n-1\}$ construct linear maps between consecutive $v_i$ as follows. 
\begin{eqnarray*}
\mathbf{\ell}^{\;v}_i(t) \co \left[\frac{i}{n}, \frac{i+1}{n}\right] \to \R^{3} &=& (i+1-nt)v_i + (nt-i)v_{i+1}
\end{eqnarray*}
Each such map is a \emph{segment}. Let $\psi_v\co [0,1] \to \R^3$ be the union of the $\ell_i^v$, so for $t$ in the domain of $\ell_i^v$, $\psi_v(t) = \ell_i^v(t)$. Such a map is a \emph{PL map with n segments}. 
\end{defn}

\begin{defn} Let $v\in ({\rm Int}\;\I^3)^{n-1}$ and $\psi_v$ the corresponding PL map with n segments. $\psi_v$ is \emph{non-singular} if it is injective, in which case we call it a \emph{PL knot with n segments}.
Denote by $L_n$ space of all PL knots with $n$ segments. 
\end{defn}

The maps which make $L_n$ into a directed system are constructed in a manner similar to the maps between the $K_n$, but are much more straightforward because there is no corresponding notion of a knot with a segment of zero length. (As we do not intend to study singular PL curves, we omit discussion pertaining to those curves.) Explicitly, for each $n$ we define a map which inserts a vertex at the midpoint of the $\alpha(n)$th segment (for some appropriate distribution $\alpha$) and reparameterize the map, which we accomplish simply by reconstructing the curve from the new list of vertices.

\begin{defn} Let $A$, $n$ and $\alpha$ be defined as in Definition \ref{def:iotag}. Define a map $\hat{I}_n \co ({\rm Int}\;\I^3)^{n-1} \to ({\rm Int}\;\I^3)^{n}$, by 
\begin{eqnarray*}(\hat{I}_n(v))_i&=&\left\{\begin{array}{ll}
v_i&i < \alpha\\
\frac{1}{2}(v_{i-1} + v_{i})&i = \alpha\\
v_{i-1}&i > \alpha.
\end{array}\right.
\end{eqnarray*}
and a map $I_n\co L_n \to L_{n+1}$ by $I_n(\psi_v) = \psi_{\hat{I}(v)}$.
\end{defn}

It is apparent that $I_{n}$ is injective for all $n$. 

\begin{defn}
For $m > n$, write $I_{n, m} = I_{m-1} \circ I_{m-2} \circ \dots \circ I_n$, and let $I_{n, n}$ be the identity. Write $I_{n,\infty}\co L_n\to \dirlim L_n$ for the induced map.
\end{defn}

Geometric isotopy and (topological) isotopy of PL knots are defined, in an analogous manner as for plumbers' knots, as isotopy at a fixed $L_n$ and after passing to the colimit, respectively.

Let $\mathcal{K}$ be the space of $C^1$ long knots in $\R^3$. With the construction we have exhibited, it is classically known that $\dirlim L_n \simeq \mathcal{K}$. We wish to establish that  $\dirlim K_n \simeq \dirlim L_n$ in order to use the space of plumbers' knots as a model for studying isotopy classes of knots in $\mathcal{K}$. Our method for proving this fact is standard but technically intricate: we produce maps from each directed system to the other which are, up to isotopy, inverses in the limit. 

In order to understand what our maps are doing to components in the limit, it will be convenient to be able to compare the topological isotopy types of the images of PL and plumbers' knots. It is straightforward to approximate with arbitrary precision the image of a plumbers' or PL knot with a non-singular $C^1$-knot by ``smoothing'' corners. Denote such a smoothing by $\widetilde{\phi}_v$ or $\widetilde{\psi}_v$ respectively.

\begin{defn}
Let $g\co K_n \to L_m$ (respectively $h\co  L_m \to K_n$). We say $g$ \emph{respects the isotopy type of $\phi_v$} if $\widetilde{\phi_v}$ and $\widetilde{g(\phi_v)}$ (respectively $\widetilde{\psi}_v$ and $\widetilde{h(\psi_v)}$) are isotopic as $C^1$ knots. If $g$ (or h) respects the knot type of all knots, we say \emph{it respects isotopy types}. 
\end{defn}

We would like to use the naive map that sends a plumbers' knot to a PL knot defined by the same set of vertices. However, not all plumbers' knots have vertices which produce valid PL knots, as plumbers' knots can have pipes of zero length. In order to produce continuous maps, we rely on techniques similar to those used in Section \ref{sec:maps} and ``buckle" knots which have segments whose length is below a particular threshold, as illustrated in Figure \ref{fig:buckle}. Thus, our map will be $f_n:K_n\to L_{3n}$.  

\begin{defn} \label{def:gpd}
Define the \emph{global perturbation distance} $\epsilon: K_n \to \R$ by 
\begin{equation*}
\epsilon(\phi_v) = \frac{1}{10}\min(\epsilon_1(\phi_v), \epsilon_2(\phi_v))
\end{equation*}
where
\begin{eqnarray*}
\epsilon_1(\phi_v) &=&\min\{d(\mathbf{p}, \mathbf{q}) \;|\; \mathbf{p}, \mathbf{q} \in \{\x_i^v, \y_i^v, \z_i^v\}_{i=0}^{n-1}\text{ s.t. }\mathbf{p}, \mathbf{q} \text{ are distant}.\}\\
\epsilon_2(\phi_v) &=&\min\{d(\mathbf{p}, \partial(\I^{3n-3})) \;|\; \mathbf{p} \in \{\x_i^v, \y_i^v, \z_i^v\}_{i=1}^{n-2} \}.
\end{eqnarray*}
Here, distances are measured using the images of pipes.
\end{defn}

That is, the global perturbation distance is a small fraction of the minimum of the distances between distant pipes and between pipes which do not intersect the boundary of the unit cube and that boundary. By moving each point in the image of the knot more than this distance, we are certain that the resulting map will respect isotopy types.

\begin{defn}\label{def:buckling}
Let $\phi_v \in K_n$ and let $\mathbf{p}$ be a pipe in $\phi_v$. Define the \emph{buckling function} for $\mathbf{p}$ by
\begin{eqnarray*}
\beta(\mathbf{p})&=&\left\{\begin{array}{ll}
\sqrt{\frac{\epsilon(\phi_v)^2 - ||\mathbf{p}||^2}{2}}&\epsilon > ||\mathbf{p}||\\
0& \epsilon \leq ||\mathbf{p}||. 
\end{array}\right.\\
\end{eqnarray*}
\end{defn}

The buckling function acts as the borrowing function did before, providing an interpolation which allows us to make our function continuous. As the length of a particular pipe shrinks, we deform the the image of the knot in $L_{3n}$ by moving the vertices at its endpoints into the pipes which neighbor it, as in Figure \ref{fig:buckle}. 

We can now give the map $f_n\co  K_n \to L_{3n}$ obtained by buckling a plumbers' knot. 

\begin{defn}\label{def:fn}
Let $\phi_v \in K_n$. For each $i\in \{0,\dots,n-1\}$ define
\begin{eqnarray*}
(\hat{f}_n( v))_{3i}&=&({v}^x_i+ s(\x_i)||\x_i||\beta(\z_{i-1}),{v}^y_i,{v}^z_i-s(\z_{i-1})||\z_{i-1}||\beta(\x_i)),\\
(\hat{f}_n( v))_{3i+1}&=&({v}^x_i-s(\x_i)||\x_i||\beta(\y_i),{v}^y_i+s(\y_i)||\y_i||\beta(\x_i),{v}^z_i),\\
(\hat{f}_n( v))_{3i+2}&=&({v}^x_i,{v}^y_i- s(\y_i)||\y_i||\beta(\z_i),{v}^z_i+s(\z_i)||\z_i||\beta(\y_i)).
\end{eqnarray*} 
where for notational convenience, we define $\beta(\z_{-1}) = ||\z_{-1}|| = 0$, and let $f_n(\phi_v) = \psi_{\hat{f_n}(v)}$.
\end{defn}

\begin{figure} 
\begin{center}
\psfrag{C1}{$v_{i}$}
\psfrag{C2}{$v_{i+1}$}
\psfrag{C3}{$\hat{f}_n(v)_{3i}$}
\psfrag{C4}{$\hat{f}_n(v)_{3i+1}$}
\psfrag{C5}{$\hat{f}_n(v)_{3i+2}$}
\psfrag{C6}{$\hat{f}_n(v)_{3i+3}$}
\includegraphics[width=7cm]{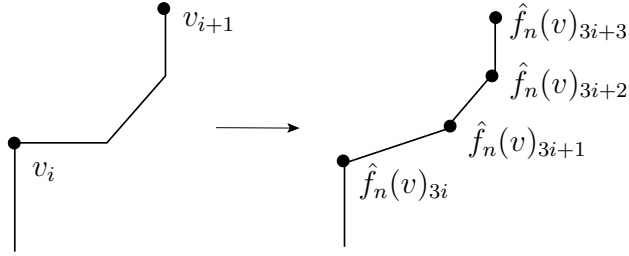}
\caption{Buckling a plumbers' knot to create a PL knot}
\label{fig:buckle}
\end{center}
\end{figure}

The global perturbation distance ensures that
\begin{lem}
$f_n$ respects isotopy types.
\end{lem}

We will make use of the maps $f_n$ in our proof of Theorem \ref{thm:componentscorrect}. We also require maps in the other direction.

\begin{defn} \label{def:an}
Let $w\in L_n$. Define $A_n\co L_n\to K_n$ to be the map which takes the PL knot $\psi_w$ to the plumbers' knot on the same vertex set, $\phi_w$.
\end{defn}

In general, $A_n$ does not preserve topological knot types or even non-singularity. However, in order to ensure that $A_n$ respects knot type it is sufficient to force the segments in $\psi_w$ to be ``short'', and we can accomplish this for a given knot $\psi_w$ by looking at its image under $I_{n, N}$ for large enough $N$. To find such $N$, we require the following measurement.

\begin{defn}
Define the \emph{maximal segment length} function $\delta \co L_n \to \mathbb{R}$ by
\begin{equation*}
\delta(\psi_w)=\frac{1}{10}\min(m_1(\psi_w), m_2(\psi_w)).
\end{equation*}
where
\begin{eqnarray*}
m_1(\psi_w)&=&\min\{d(\ell_i^w, \ell_j^w)\;|\;|i - j| > 1\}\\
m_2(\psi_w)&=&\min\{d(\ell^w_i, \partial(\I^{3n-3})\;|\;i\in \{1,\dots,n-1\}\}
\end{eqnarray*}
\end{defn}

This definition is very similar to that of the global perturbation distance in Definition \ref{def:gpd}. We again wish to deform the image of the knot without changing its topological isotopy type. Specifically, here we think of a segment as the diagonal of a rectangular prism in $\R^3$, and wish to replace the segment with a plumbers' move along the boundary of that prism. If the segment is shorter than the minimal segment length for the knot, then no other edge of the knot passes through the prism, and so this replacement cannot change its topological isotopy type.

\begin{defn}
Let $N(\psi_w)$ be the minimum number so that $\psi_{I_{n, N(\psi_w)}(\psi_w)}$ is a PL knot with the property that no segment has length greater than $\delta(w)$.
\end{defn} 
Clearly, $N(\psi_w)$ is exists and is well-defined, and we have

\begin{lem} $A_{N(\psi_w)} \circ I_{n, N(\psi_w)}$ respects the isotopy type of $\psi_w$.
\end{lem}

\medskip

We finally have sufficient machinery to prove that $\dirlim K_n$ has the proper weak homotopy type. 

\begin{thm} \label{thm:htpytypecorrect}
$\dirlim K_n$ is homotopy equivalent to $\dirlim L_n$.
\end{thm}

However, the technique of the proof is nearly identical in form and content to the proof for $\pi_0$ isomorphism, requiring only more cumbersome notation to we work with families rather than individual maps and the application of compactness to bound the articulation necessary to realize a family, so we only explicitly prove

\begin{thm} \label{thm:componentscorrect} The induced map $f_*\co \pi_0(\dirlim K_n) \to \pi_0(\dirlim L_n)$ is an isomorphism of sets.
\end{thm}

\begin{proof}

We begin by showing surjectivity of $f_\ast$. Let $\bar{\psi}_w\in \dirlim L_n$ and $\psi_w \in L_k$ so that $I_{k,\infty}(\psi_w) = \bar{\psi}_w$. Write $N = N(\psi_w)$. We require that in following family of diagrams indexed by $\bar{w}$, the images of $\psi_w$ under both paths are geometrically isotopic.
\begin{eqnarray*}
 \xymatrix{
L_k
\ar@/^1pc/[rr]^{I_{k, 3N}}
\ar[r]_{I_{k, N}}
& L_{N}
 \ar[d]_{A_{N}}
& L_{3N} \\
& K_{N}
\ar[ur]_{f_{N}} } 
\end{eqnarray*}
However, as $A_N$, $f_N$ and $I_m$ for each $m$ respect the isotopy type of $\psi_w$, there is a geometric isotopy between the images. Since the $\iota_n$ are injective, this says that for $\bar{\psi}_w \in \dirlim L_n$  there exists $\bar{\phi}_v \in \dirlim K_n$ which maps to the isotopy class of $\bar{\psi}_{w}$ under $f_*$.

Now, let $\bar{\psi}_{w}$ and $\bar{\psi}_{w'}$ be isotopic elements of $\dirlim L_n$. We can lift an isotopy between them to a geometric isotopy at some finite stage, $\Psi_{\psi_w, \psi_{w'}} \co \I \to L_n$. Let $N = \max\{N(\Psi(t)) \;|\; t\in \I\}$, where $N(\psi_w)$ is as in Definition \ref{def:an}. Precompose $\Psi_{\psi_w,\psi_{ w'}}$ by $I_{n, N}$ to produce a geometric isotopy $\hat{\Psi}$ between $I_{n,N}(\psi_{w})$ and $I_{n,N}(\psi_{w})$. Now we can apply $A_N$ to get a geometric isotopy between $(A_N \circ I_{n,N})(\phi_{w})$ and $(A_N\circ I_{n,N})(\phi_{w'})$. Per the proof of surjectivity, under $f_N$ these map to elements geometrically isotopic to $I_{n,3N}(\psi_{w})$ and $I_{n,3N}(\psi_{w'}$ respectively. That is, if $\bar{\psi}_{{w}}, \bar{\psi}_{w'} \in \dirlim L_n$ are isotopic, we can construct an isotopy between elements of $\dirlim K_n$ which map to knots isotopic to $\bar{\psi}_{{w}}$ and $\bar{\psi}_{{w}'}$, so $f_*$ is injective, and thus an isomorphism.

\end{proof}

\section{Relationships with lattice knots and cube diagrams}

We observe that plumbers' knots bear strong resemblance to a number of other discrete knot theories. Two of particular interest are lattice knots and cube diagrams. Lattice knots are studied because they can be used to model physical data like length and thickness of the material from which a knot is constructed. Cube diagrams are used in \cite{Cube} to construct chain complexes of knots with which one can study knot Floer homology. 

\subsection{Lattice knots}
A lattice knot is a PL knot whose segments lie parallel to the coordinate axes and meet one another on points of the integer lattice $\Z^3 \subseteq \R^3$. Clearly, such knots are very closely related to the representative knots of Definition \ref{def:repcurve}. In fact, any representative of a cell of non-singular plumbers' knots can be ''closed" to produce a lattice knot, as in Figure \ref{fig:closeknot}.

\begin{figure} 
\begin{center}
\includegraphics[width=7cm]{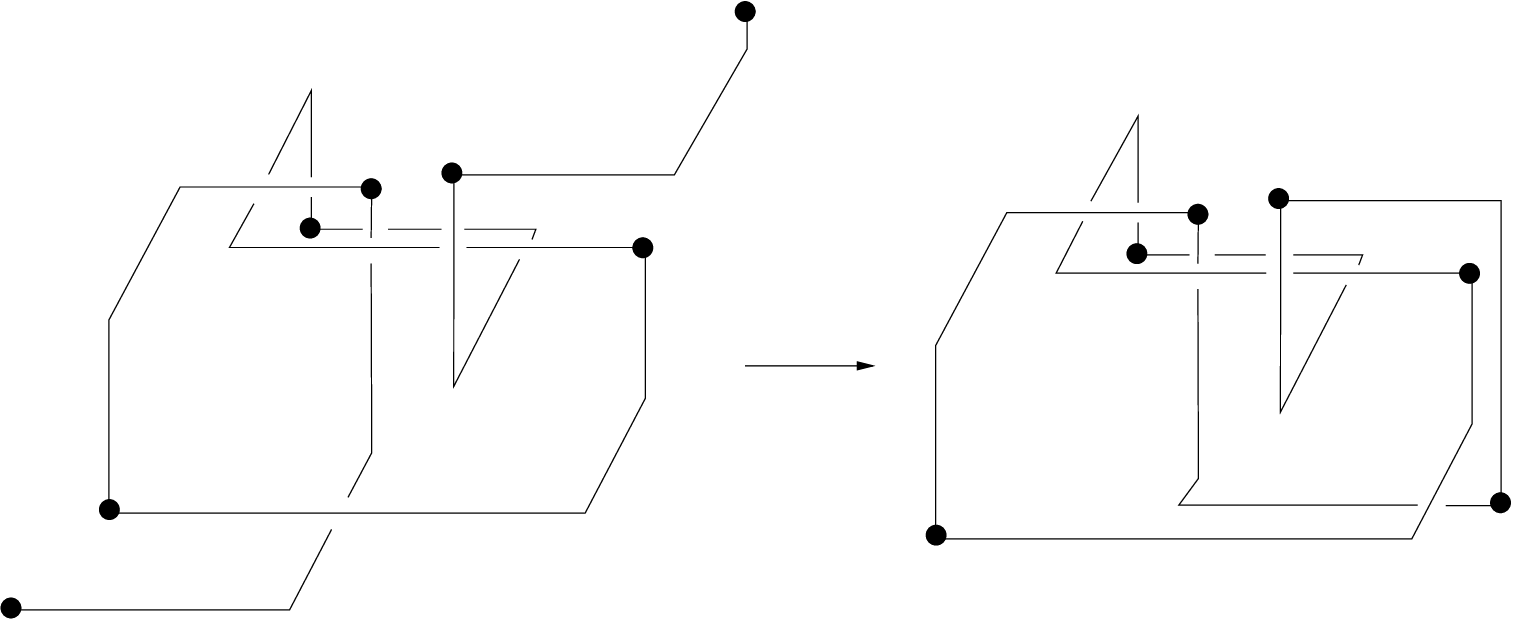}
\caption{Closing a plumbers' knot to obtain a lattice knot.}
\label{fig:closeknot}
\end{center}
\end{figure}

Recall that $L_m$ is the space of $m$-segment piecewise linear knots and suitably modify the definitions given so that maps in $L_m$ have their images in $[0,N]^3$ for some large $N$. Let $Lat_m\subseteq L_m$ be the subspace of lattice knots. 

\begin{defn} \label{def:plumtolat} Let $\phi_v\in P_n$ be a representative knot for the cell $e(\sigma_x, \sigma_y, \sigma_z)$ and define a lattice knot $LK(\phi_v)\in Lat_{3n-1}$ by

\begin{equation*}
\begin{array}{rcl}
LK(v)_0 &=& \left( n+1, n\cdot v_{n-1}^y, 0 \right)\\
LK(v)_1 &=& \left( n\cdot v_1^x, n\cdot v_{n-1}^y, 0 \right )\\
LK(v)_2 &=& \left( n\cdot v_1^x, n\cdot v_1^y, 0 \right )\\
\\
LK(v)_{3k} &=&\left( n\cdot v_{k}^x, n\cdot v_k^y, n\cdot v_k^z \right) \\
LK(v)_{3k+1} &=&\left( n\cdot v_{k+1}^x, n\cdot v_{k}^y, n\cdot v_k^z \right) \\
LK(v)_{3k+2} &=&\left( n\cdot v_{k+1}^x, n\cdot v_{k+1}^y, n\cdot v_{k}^z \right) \\
\\
LK(v)_{3n-3} &=&\left( n\cdot v_{n-1}^x, n\cdot v_{n-1}^y, n\cdot v_{n-1}^z \right) \\
LK(v)_{3n-2} &=&\left( n+1, n\cdot v_{n-1}^y, n\cdot v_{n-1}^z \right) \\
\end{array}
\end{equation*}

where $k$ ranges from $1$ to $n-2$.

\end{defn} 

Each representative knot for a codimension 0 cell of $P_n$ maps to a lattice knot with $3n-1$ segments. However, these knots tend to use more segments than necessary and there is interest in discovering the minimal number of segments required to create a lattice knot of a given topological knot type. Recall that, as in Definition \ref{def:repcurve}, the idea of a representative knot is extensible to cells of any dimension in $\Cell_\bullet(K_n)$. As some of the plumbers' knots what appear in these cells contain zero-length pipes which must be omitted from their image in $Lat_m$, we can study them as a means to find lattice knots with fewer segments.  A pipe has zero length precisely when the two vertices which define its move are in the same appropriate coordinate plane.

\begin{prop} Let $\rho\in \Sigma_n$ and define $\mu(\rho)$ to be the number of pairs of consecutive integers which appear in the same cycle in $\rho$. Let $v\in e(\sigma_x, \sigma_y, \sigma_z; \rho_x, \rho_y, \rho_z)$. The number of zero-length pipes in $\phi_v$ is $\mu(v) = \mu(\rho_x) + \mu(\rho_y) + \mu(\rho_z)$.
\end{prop}

Also, notice that when we close an $n$-move plumbers' knot for which $v_1^y = v_{n-1}^y$, there is a zero length segment produced in the closure. This occurs when $1$ and $n-1$ appear in the same cycle in $\rho_y$.

Further, since adjacent pipes can move along the same coordinate axis, it is permissible to omit a vertex in the middle of the segment when we construct the lattice knot.

\begin{prop} Let $\rho\in \Sigma_n$ and let $\nu(x)$ be the number of pairs of consecutive integers which appear in the same cycle in $\rho_y$ and $\rho_z$. Let $v\in e(\sigma_x, \sigma_y, \sigma_z; \rho_x, \rho_y, \rho_z)$. The number of consecutive moves which travel only along single axes in $\phi_v$ is $\nu(v) = \nu(x) + \nu(y) + \nu(z)$.
\end{prop}

Notice that if $\phi_v$ is a plumbers' knot, the same pair of consecutive integers can never occur in all three of $\rho_x$, $\rho_y$ and $\rho_z$, as this would produce three consecutive zero-length pipes.

\begin{lem} \label{lem:latbound} For $v\in K_n$, $\mu(v) + \nu(v) + \delta_{v_1^y, v_{n-1}^y} \leq 3(n-2)+1$.
\end{lem}

\begin{defn} Fix a cell $e(\sigma_x, \sigma_y, \sigma_z; \rho_x, \rho_y, \rho_z)$ and let $\phi_v\in K_n$ be the representative knot for $e$. Let $\mu(v)$ and $\nu(v))$ be as above. Define a lattice knot $LK(v)\in Lat_{3n-1-\mu(v) - \nu(v)}$ in the same manner as in Definition \ref{def:plumtolat}, omitting vertices which would coincide or which bound two segments which move along the same coordinate axis.
\end{defn} 

Using this definition, Lemma \ref{lem:latbound} says that the smallest number of segments that can occur in a lattice knot arising as the closure of a plumbers' knot of $n$ moves is 4. Clearly, such a lattice knot is an unknot. This bound is acheived by the representative of the cell $e(12\dots (n\!\!-\!\!1)_x,12\dots (n\!\!-\!\!1)_y,12\dots (n\!\!-\!\!1)_z;(1\;2\dots \;(n\!\!-\!\!1))_y,(1\;2\dots \;(n\!\!-\!\!1))_z)$, for example.

The lattice knots which are produced by plumbers' knots are characterized by one or two segments lying in the z=0 plane and one in the plane with the highest $x$ coordinate. Every lattice knot can be deformed to such, under the appropriate notion of isotopy. It seems likely, therefore, that plumbers' knots will serve as a bridge to allow translation of tools from classical knot theory to be applied to finite complexity knot theory and vice versa.

\subsection{Remarks on cube diagrams}

The cube diagrams of Baldridge and Lowrance \cite{Cube} bear strong resemblance to lattice knots and can be considered as plumbers' knot representatives of particular cells. 

\begin{prop} Let $e=e(\sigma_x, \sigma_y, \sigma_z) \in \Cell_{3n-3}(K_n)$ and $v\in e$. The plumbers' knot $\phi_v$ satisfies the \emph{x-y crossing condition} described in \cite{Cube} if whenever $\sigma_x^{-1}(b+1)$ is between $\sigma_x^{-1}(a)$ and $\sigma_x^{-1}(a+1)$ and $\sigma_y^{-1}(a)$ is between $\sigma_y^{-1}(b)$ and $\sigma_y^{-1}(b+1)$, then $\sigma_z^{-1}(b) > \sigma_z^{-1}(a)$. (Compare to Definition \ref{def:produceint}.)
\end{prop}

This follows immediately from the definitions, and symmetric statements exist for the y-z and z-x crossing conditions. We can consider the subspace of cube knots of $n$ moves, $C_n \subseteq K_n$, generated by such cells, which are precisely the cube knots of grid number $n$. If we do not allow stabilization moves, isotopy of cube knots in each finite space is the same as geometric isotopy through cube knots. Application of the algorithm for classification of plumbers' knots yields that $C_5$ has precisely one cell, a right-handed trefoil, while $C_6$ has 11 components and $C_7$ has 108.

Allowing for the stabilization moves described in \cite{Cube}, the authors prove the following analogue to Theorem \ref{thm:componentscorrect}.

\begin{thm}[\cite{Cube}]
$\pi_0(\underrightarrow\lim C_n) \cong \pi_0(\underrightarrow\lim K_n)$
\end{thm}

As cube diagrams can be used to compute combinatorial knot Floer homology via a modification of the grid diagram algorithm in \cite{CombFloer}, these subspaces $C_n$ are of particular interest. We expect that our development of the combinatorics of plumbers' knots will illuminate computations in the cube diagram chain complex for knot Floer homology. Further, our study of plumbers' knots relationship to finite-type invariants in \cite{UnstabVas} may provide a method of understanding connections between the two theories.

%
%
%
\bibliographystyle{plain}
\bibliography{../central_bibliography}

\end{document}